\UseRawInputEncoding
\documentclass[11pt,a4paper]{article}
\usepackage{amsfonts}
\usepackage{amssymb}
\usepackage{mathrsfs}
\usepackage{amsmath}
\usepackage{booktabs}
\usepackage{epsf,epsfig,amsfonts,amsgen,indentfirst}
\usepackage{amsmath,amstext,amsbsy,amsopn,amsthm,bbding,wasysym}
\usepackage{multicol,mathdots}
\usepackage{subfigure}
\usepackage[numbers,sort&compress]{natbib}
\allowdisplaybreaks

\newtheorem{theorem}{Theorem}[section]

\newtheorem{lemma}[theorem]{Lemma}
\newtheorem{cor}[theorem]{Corollary}

\newtheorem{defn}[theorem]{Definition}

\newtheorem{Conjecture}{Conjecture}

\newtheorem{Problem}{Problem}

\newtheorem{example}{Example}

\begin{document}
\title
{\LARGE \textbf{On the edge reconstruction of the second immanantal polynomials of undirected graph and digraph\thanks{ supported by NSFC (No. 12261071) and NSF of Qinghai Province (No. 2020-ZJ-920).} }}

\author{ Tingzeng Wu$^{a,b}$\thanks{{Corresponding author.\newline
\emph{E-mail address}: mathtzwu@163.com, g1617069052@163.com, hjlai@math.wvu.edu
}}, Yafan Geng$^a$, Hong-Jian Lai$^c$\\
{\small $^{a}$ School of Mathematics and Statistics, Qinghai Nationalities University, }\\
{\small  Xining, Qinghai 810007, P.R.~China} \\
{\small $^{b}$ Qinghai Institute of Applied Mathematics,   Xining, Qinghai, 810007, P.R.~China} \\
{\small $^{c}$ Department of Mathematics, West Virginia University, Morgantown, WV, USA} }

\date{}

\maketitle
\noindent {\bf Abstract:} Let $M=(m_{ij})$ be an $n\times n$ matrix. The second immanant of matrix $M$ is defined by
\begin{eqnarray*}
d_{2}(M)=\sum_{\sigma\in S_{n}}\chi_{2}(\sigma)\prod_{s=1}^{n}m_{s\sigma(s)},
\end{eqnarray*}
where $\chi_{2}$ is the irreducible character of $S_{n}$ corresponding to the partition $(2^{1},1^{n-2})$. The polynomial $d_{2}(xI-M)$ is called the second immanantal polynomial of matrix $M$. Denote by $D(G)$ (resp. $D(\overrightarrow{G})$) and $A(G)$ (resp. $A(\overrightarrow{G})$) the diagonal matrix of vertex degrees and the adjacency matrix of undirected graph $G$ (resp. digraph $\overrightarrow{G}$), respectively. In this article, we prove that $d_{2}(xI-A(G))$ (resp. $d_{2}(xI-A(\overrightarrow{G}))$) can be reconstructed from the second immanantal polynomials of the adjacency matrix of all subgraphs in $\{G-uv,G-u-v|uv\in E(G)\}$ (resp. $\{\overrightarrow{G}-e|e\in E(\overrightarrow{G})\}$).  Furthermore, the polynomial $d_{2}(xI-D(\overrightarrow{G})\pm A(\overrightarrow{G}))$ can also be reconstructed by the second immanantal polynomials of the (signless) Laplacian matrixs of all subgraphs in $\{\overrightarrow{G}-e|e\in E(\overrightarrow{G})\}$, respectively.


\section{Introduction}

Let $S_{n}$ be the permutation group on $n$ symbols and $\lambda$ be a partition of $n$, and let $\chi_{\lambda}$ be the irreducible character of $S_{n}$ corresponding to the partition $\lambda$. The {\em immanant function} $d_{\lambda}$ associated with the character $\chi_{\lambda}$ acting on an $n\times n$ matrix $M=[m_{ij}]$ is defined as
\begin{eqnarray*}
d_{\lambda}(M)=\sum_{\sigma\in S_{n}}\chi_{\lambda}(\sigma)\prod_{i=1}^{n}a_{i\sigma(i)}.
\end{eqnarray*}
\noindent B{\"u}rgisser \cite{1} demonstrated that computation of immanant is VNP-complete. The determinant,  second immanant and permanent are the immanants corresponding to the $\lambda=(1^{n})$, $\lambda=(2^{1},1^{n-2})$ and $\lambda=(n)$, respectively.  Let $X=(x_{ij})$ be an $k\times k$ matrix $(k\geq2)$, and let $X(i)$ be the submatrix of $X$ obtained by deleting row and column $i$. The determinant and the second immanant of $X$ are denoted by $\det(X)$ and $d_{2} (X)$, respectively. Then we have
\begin{eqnarray}\label{equ1}
d_{2}(X)=\sum_{i=1}^{k}x_{ii}\det(X(i))-\det(X).
\end{eqnarray}

The {\em immanantal polynomial} of an $n\times n$ matrix $M$ associated with the character $\chi_{\lambda}$ is denoted by
\begin{eqnarray*}
d_{\lambda}(xI-M)=\sum_{k=0}^{n}(-1)^{k}c_{\lambda,k}(M)x^{n-k}.
\end{eqnarray*}

Let $G=(V(G),E(G))$ be an undirected graph without loops or multiple edges, with vertex set $V(G)=\{v_{1},v_{2},\ldots,v_{n}\}$ and edge set $E(G)=\{e_{1},e_{2},\ldots,e_{m}\}$. The number of vertices and edges of $G$ denoted by $|V(G)|=n$ and $|E(G)|=m$, respectively. The adjacency matrix of graph $G$ denoted by $A(G)=(a_{ij})(i,j\in{1,2,\cdots,n})$, where
$$
a_{ij}=
  \begin{cases}
   1, & \text{if $v_{i}$ and $v_{j}$ are adjacent},\\
   0, & \text{otherwise}.
  \end{cases}
$$

Let $\overrightarrow{G}=(V(\overrightarrow{G}),E(\overrightarrow{G}))$ be an digraph without loops or multiple arcs, with vertex set $V(\overrightarrow{G})=\{v_{1},v_{2},\ldots,v_{n}\}$ and arc set $E(\overrightarrow{G})=\{\overrightarrow{e_{1}},\overrightarrow{e_{2}},\ldots,\overrightarrow{e_{m}}\}$. The number of vertices and edges of $\overrightarrow{G}$ denoted by $|V(\overrightarrow{G})|=n$ and $|E(\overrightarrow{G})|=m$, respectively. The adjacency matrix of graph $\overrightarrow{G}$ denoted by $A(\overrightarrow{G})=(a_{ij})_{n\times n}$, where
$$
a_{ij}=
  \begin{cases}
   1, & \text{if $(v_{i},v_{j})\in E(\overrightarrow{G})$},\\
   0, & \text{otherwise}.
  \end{cases}
$$
The vertex in-degree diagonal matrix of $\overrightarrow{G}$ denoted by $D(\overrightarrow{G})=diag(d^{-}(v_{1}),d^{-}(v_{2}),\ldots,d^{-}(v_{n}))$, where $d^{-}(v_{i})$ is the in-degree of $v_{i}$. Then $D(\overrightarrow{G})-A(\overrightarrow{G})$ and $D(\overrightarrow{G})+A(\overrightarrow{G})$ are the Laplacian and the signless Laplacian matrix of graph $\overrightarrow{G}$, respectively.

Denote by $\tau(G;x)$ the second immanantal polynomial of the adjacency matrix of an undirected graph $G$. Then
\begin{eqnarray}\label{equ2}
\tau(G;x)=d_{2} (xI_{n}-A(G)),
\end{eqnarray}
where $I_{n}$ is an $n\times n$ identity matrix. And its theory is well elaborated \cite{RM, WYG, WYFG}.

Let $g_{1}(\overrightarrow{G};x),g_{2}(\overrightarrow{G};x)$ and $g_{3}(\overrightarrow{G};x)$ be the second immanantal polynomials of the adjacency matrix and second immanantal polynomial of the (signless) Laplacian matrix of digraph $\overrightarrow{G}$, respectively. Thus
\begin{eqnarray*}
g_{1}(\overrightarrow{G};x)&=&d_{2} (xI_{n}-A(\overrightarrow{G})),\\
g_{2}(\overrightarrow{G};x)&=&d_{2} (xI_{n}-D(\overrightarrow{G})+A(\overrightarrow{G})),\\
g_{3}(\overrightarrow{G};x)&=&d_{2} (xI_{n}-D(\overrightarrow{G})-A(\overrightarrow{G})).
\end{eqnarray*}

The Ulam-Kelly's vertex reconstruction conjecture \cite{2,3} and Harary's edge reconstruction conjecture \cite{4,5} are long lasting problems in graph theory, as stated below.
\begin{Conjecture}\label{con1}(Ulam-Kelly's vertex reconstruction conjecture \cite{3})
Every simple graph $G$ with vertex set $V(G)$ can be uniquely reconstructed from its vertex deck $\{G-v|v\in V(G)\}$ for $|V(G)|\geqslant 3$.
\end{Conjecture}
\begin{Conjecture}\label{con1}(Harary's edge reconstruction conjecture \cite{4})
Every simple graph $G$ with edge set $E(G)$ can be reconstructed from its edge deck $\{G-e|e\in E(G)\}$ for $|E(G)|\geqslant 4$.
\end{Conjecture}

\noindent Although many results for these two conjectures have been obtained \cite{6,7,8,9}, they are still open. As a spectral counter part to Ulam-Kelly's vertex reconstruction conjecture, Cvetkovi\'{c} posed a polynomial reconstruction problem: When $|V(G)|\geqslant 3$, can the characteristic polynomial of a simple graph $G$ with vertex set $V(G)$ be reconstructed from the characteristic polynomial of all subgraphs in $\{G-v|v\in V(G)\}$? The same problem is also posed by Schwenk \cite{10}. Gutman and Cvetkovi\'{c} \cite{11} obtained some results of this problem. Zhang et al. \cite{12,13} studied the corresponding polynomial edge reconstruction problem: Can the characteristic polynomial of a simple graph $G$ with edge set $E(G)$ be reconstructed from the characteristic polynomial of all subgraphs in $\{G-e|e\in E(G)\}$? They proved that if $|V(G)|\neq |E(G)|$, then the characteristic polynomial (resp. permanental polynomial) of a simple graph $G$ with vertex set $V(G)$ and edge set $E(G)$ can be reconstructed from the characteristic polynomial (resp. permanental polynomial) of all subgraphs in $\{G-uv,G-u-v|uv\in E(G)\}$. They also proved that similar results hold for the Laplacian (resp. signless Laplacian) characteristic polynomial and the Laplacian (resp. signless Laplacian) permanental polynomial of digraph $\overrightarrow{G}$ if $|V(\overrightarrow{G})|\neq |E(\overrightarrow{G})|$. The characteristic polynomial,  permanental polynomial and  second immanantal polynomial are special cases of immanantal polynomials. There exist naturally two problems as follows.

\begin{Problem}\label{prob1}
Can the second immanantal polynomial of undirected graph $G$ satisfy the edge reconstruction conjecture?
\end{Problem}

\begin{Problem}\label{prob2}
 Can the second immanantal polynomial of digraph $\overrightarrow{G}$ satisfy the edge reconstruction conjecture?
\end{Problem}

The current research is motivated by the two problems above. This paper is organized as follows. In the next section, we settle Problem 1 by showing that if $m\neq n$, $\tau(G;x)$ can be reconstructed from $\{\tau(G-v_{s}v_{t};x),\tau(G-v_{s}-v_{t};x)|v_{s}v_{t}\in E(G)\}$. In Section 3, we address Problem 2, and prove that if $m\neq n$, $g_{i}(\overrightarrow{G};x)$ can be reconstructed from $\{g_{i}(\overrightarrow{G}-\overrightarrow{e};x)|\overrightarrow{e}\in E(\overrightarrow{G})\}$ for any $i=1,2,3$. In Section 4, a brief summary is presented and two future research problems are proposed.

\section{A solution of Problem \ref{prob1}}\label{sec2}

In this section, a solution to the Problem \ref{prob1} will be presented. Some notation and lemmas are of importance to the description and proof of our results later, and we list them below.

Let $X$ be an $n\times n$ matrix whose rows and columns are labeled with elements in $\{1,2,\ldots,n\}$. And let $X^{a_{1},\ldots,a_{i}}_{b_{1},\ldots,b_{j}}$ be the submatrix of $X$ obtained by deleting all rows in with labels $\{a_{1},\ldots,a_{i}\}$ and all columns in with labels $\{b_{1},\ldots,b_{j}\}$, where $i,j\in\{1,\ldots,n\}$.

\begin{defn}\label{def2.1}
Let $B=(b_{st})_{n\times n}$ be a symmetric matrix of order $n$ over the complex field, hence $b_{st}=b_{ts}(1\leqslant s,t\leqslant n)$. Define a symmetric matrix $B_{[ij]}=(b_{st}^{ij})_{n\times n}$, where $1\leqslant i,j\leqslant n$, and
$$
b_{st}^{ij}=
  \begin{cases}
   b_{st}, & \text{if $(s,t)\neq(i,j)$ and $(s,t)\neq(j,i)$},\\
   0, & \text{otherwise}.
  \end{cases}
$$
\end{defn}
\begin{example}\label{exam1}
Let $B$ be an $n\times n$ matrix as follows:
$$
B=
\left(
  \begin{array}{cccc}
    b_{11} & b_{12} &  \cdots &  b_{1n} \\
    b_{12} & b_{22} &  \cdots &  b_{2n} \\
    \vdots & \vdots &  \ddots &  \vdots \\
    b_{1n} & b_{2n} &  \cdots &  b_{nn}
  \end{array}
\right).
$$
Then
$$
B_{[12]}=B_{[21]}=
\left(
  \begin{array}{cccc}
    b_{11} & 0 &  \cdots &  b_{1n} \\
    0 & b_{22} &  \cdots &  b_{2n} \\
    \vdots & \vdots &  \ddots &  \vdots \\
    b_{1n} & b_{2n} &  \cdots &  b_{nn}
  \end{array}
\right),
B_{[11]}=
\left(
  \begin{array}{cccc}
    0 & b_{12} &  \cdots &  b_{1n} \\
    b_{12} & b_{22} &  \cdots &  b_{2n} \\
    \vdots & \vdots &  \ddots &  \vdots \\
    b_{1n} & b_{2n} &  \cdots &  b_{nn}
  \end{array}
\right).
$$
\end{example}
\noindent By definition, if $b_{ij}=0$, then $B=B_{ij}=B_{ji}$. Zhang et al. \cite{12} gave a result as follows. For any $1\leqslant s\leqslant n,1\leqslant i\neq j\leqslant n$, then
\begin{eqnarray}\label{equ3}
\det(B_{[ss]})&=&\det(B)-b_{ss}\det(B_{s}^{s}),
\end{eqnarray}
\begin{eqnarray}\label{equ4}
\det(B_{[ij]})&=&\det(B)-(-1)^{i+j}b_{ij}\det(B_{j}^{i})-(-1)^{i+j}b_{ji}\det(B_{i}^{j})-b_{ij}^{2}\det(B_{i,j}^{i,j}).
\end{eqnarray}

\begin{lemma}\label{lem2.2}(\cite{12})
Let $B=(b_{st})_{n\times n}$ and $B_{[ij]}=(b_{st}^{ij})_{n\times n}$ be defined as in Definition \ref{def2.1}. Then the determinant of $B$ satisfies:
\begin{eqnarray*}
\frac{1}{2}(n^{2}-n)\det(B)=\sum_{1\leqslant i\leqslant j\leqslant n}\det(B_{[ij]})+\sum_{1\leqslant i< j\leqslant n}b_{ij}^{2}\det(B_{i,j}^{i,j}).
\end{eqnarray*}
\end{lemma}

There exists a relation between $d_{2}(B)$ and $d_{2}(B_{[ij]})$ as follows.

\begin{lemma}\label{lem2.3}
Let $B=(b_{st})_{n\times n}$ and $B_{[ij]}=(b_{st}^{ij})_{n\times n}$ be defined as in Definition \ref{def2.1}. Then the second immanant of $B$ and $d_{2}(B_{[ij]})$ satisfies:
\begin{eqnarray*}
\frac{1}{2}(n^{2}-n)d_{2}(B)=\sum_{1\leqslant i\leqslant j\leqslant n}d_{2}(B_{[ij]})+\sum_{1\leqslant i< j\leqslant n}b_{ij}^{2}d_{2}(B_{i,j}^{i,j}).
\end{eqnarray*}
\end{lemma}

\begin{proof}
Denote by $b_{uu}^{*}$ the $(u,u)$-entry in matrix $B_{[ij]}$. By (\ref{equ1}) and Lemma \ref{lem2.2}, we have
\begin{eqnarray}\label{equ5}
&~&\sum_{1\leqslant i\leqslant j\leqslant n}d_{2}(B_{[ij]})\notag \\
&=&\sum_{1\leqslant i\leqslant j\leqslant n}\sum_{u=1}^{n}b_{uu}^{*}\det((B_{[ij]})_{u}^{u})-\sum_{1\leqslant i\leqslant j\leqslant n}\det(B_{[ij]})\notag \\
&=&\sum_{1\leqslant i\leqslant j\leqslant n}\sum_{u=1}^{n}b_{uu}^{*}\det((B_{[ij]})_{u}^{u})-(\frac{n^{2}-n}{2})\det (B)+\sum_{1\leqslant i<j\leqslant n}b_{ij}^{2}\det(B_{i,j}^{i,j}),
\end{eqnarray}
\begin{eqnarray}\label{equ6}
\sum_{1\leqslant i\leqslant j\leqslant n}\sum_{u=1}^{n}b_{uu}^{*}\det((B_{[ij]})_{u}^{u})=\sum_{i=1}^{n}\sum_{u=1}^{n}b_{uu}^{*}\det((B_{[ii]})_{u}^{u})
+\sum_{1\leqslant i< j\leqslant n}\sum_{u=1}^{n}b_{uu}^{*}\det((B_{[ij]})_{u}^{u}).
\end{eqnarray}
For simplicity of presentation, define
$$
\Delta_{1}=\sum_{i=1}^{n}\sum_{u=1}^{n}b_{uu}^{*}\det((B_{[ii]})_{u}^{u})~{\rm and}~
\Delta_{2}=\sum_{1\leqslant i< j\leqslant n}\sum_{u=1}^{n}b_{uu}^{*}\det((B_{[ij]})_{u}^{u}).
$$
Thus (\ref{equ6}) becomes $\sum_{1\leqslant i\leqslant j\leqslant n}\sum_{u=1}^{n}b_{uu}^{*}\det((B_{[ij]})_{u}^{u})=\Delta_{1}+\Delta_{2}$. We note that if $u=i$ then $\Delta_{1}=0$. Assume that $u\neq i$. By (\ref{equ3}), we have
\begin{eqnarray*}
\Delta_{1}&=&\sum_{i=1}^{n}\sum_{u=1}^{n}b_{uu}\det(B_{u}^{u})-\sum_{i=1}^{n}\sum_{u=1}^{n}b_{uu}b_{ii}\det(B_{u,i}^{u,i})\\
&=&(n-1)\sum_{u=1}^{n}b_{uu}\det(B_{u}^{u})-\sum_{i=1}^{n}\sum_{u=1}^{n}b_{uu}b_{ii}\det(B_{u,i}^{u,i}).
\end{eqnarray*}
For $\Delta_{2}$, we note that if $u=i$ or $j$ then $$\Delta_{2}=(n-1)\sum_{u=1}^{n}b_{uu}\det(B_{u}^{u}).$$
If $u\neq i,j$, then by (\ref{equ4}), we have
\begin{eqnarray*}
\Delta_{2}&=&\sum_{1\leqslant i<j\leqslant n}\sum_{u=1}^{n}b_{uu}\det(B_{u}^{u}) -\sum_{1\leqslant i<j\leqslant n}\sum_{u=1}^{n}b_{uu}(-1)^{i+j}b_{ij}\det(B_{u,j}^{u,i})\\
&~&-\sum_{1\leqslant i<j\leqslant n}\sum_{u=1}^{n}b_{uu}(-1)^{i+j}b_{ji}\det(B_{u,i}^{u,j})-\sum_{1\leqslant i<j\leqslant n}\sum_{u=1}^{n}b_{uu}b_{ij}^{2}\det(B_{u,i,j}^{u,i,j})\\
&=&(\frac{n^{2}-3n+2}{2})\sum_{u=1}^{n}b_{uu}\det(B_{u}^{u}) -\sum_{1\leqslant i<j\leqslant n}\sum_{u=1}^{n}b_{uu}(-1)^{i+j}b_{ij}\det(B_{u,j}^{u,i})\\
&~&-\sum_{1\leqslant i<j\leqslant n}\sum_{u=1}^{n}b_{uu}(-1)^{i+j}b_{ji}\det(B_{u,i}^{u,j})-\sum_{1\leqslant i<j\leqslant n}\sum_{u=1}^{n}b_{uu}b_{ij}^{2}\det(B_{u,i,j}^{u,i,j}).
\end{eqnarray*}
Combining arguments above, we obtain that
\begin{eqnarray}\label{equ7}
&~&\sum_{1\leqslant i\leqslant j\leqslant n}\sum_{u=1}^{n}b_{uu}^{*}\det((B_{[ij]})_{u}^{u})\notag \\ &=&(\frac{n^{2}+n-2}{2})\sum_{u=1}^{n}b_{uu}\det(B_{u}^{u})-\sum_{\substack{1\leqslant i,j\leqslant n\\i,j\neq u}} \sum_{u=1}^{n}b_{uu}(-1)^{i+j}b_{ij}\det(B_{u,j}^{u,i})\notag \\
&~&-\sum_{\substack{1\leqslant i<j\leqslant n\\i,j\neq u}}\sum_{u=1}^{n}b_{uu}b_{ij}^{2}\det(B_{u,i,j}^{u,i,j}).
\end{eqnarray}
By substituting (\ref{equ7}) into (\ref{equ5}), we obtain that
\begin{eqnarray*}
&~&\sum_{1\leqslant i\leqslant j\leqslant n}d_{2}(B_{[ij]})\\
&=&(\frac{n^{2}+n-2}{2})\sum_{u=1}^{n}b_{uu}\det(B_{u}^{u})-\sum_{\substack{1\leqslant i,j\leqslant n\\i,j\neq u}} \sum_{u=1}^{n}b_{uu}(-1)^{i+j}b_{ij}\det(B_{u,j}^{u,i})\\
&~&-\sum_{\substack{1\leqslant i<j\leqslant n\\i,j\neq u}}\sum_{u=1}^{n}b_{uu}b_{ij}^{2}\det(B_{u,i,j}^{u,i,j})-(\frac{n^{2}-n}{2})\det (B)+\sum_{1\leqslant i<j\leqslant n}b_{ij}^{2}\det(B_{i,j}^{i,j})\\
&=&(\frac{n^{2}-n}{2})d_{2}(B)+(n-1)\sum_{u=1}^{n}b_{uu}\det(B_{u}^{u})-\sum_{\substack{1\leqslant i,j\leqslant n\\i,j\neq u}} \sum_{u=1}^{n}b_{uu}(-1)^{i+j}b_{ij}\det(B_{u,j}^{u,i})\\
&~&-\sum_{1\leqslant i<j\leqslant n}b_{ij}^{2}d_{2}(B_{i,j}^{i,j})\\
&=&(\frac{n^{2}-n}{2})d_{2}(B)+(n-1)\sum_{u=1}^{n}b_{uu}\det(B_{u}^{u})-\sum_{u=1}^{n} b_{uu}\sum^{n}_{\substack{i=1\\i\neq u}}\sum^{n}_{\substack{j=1\\j\neq u}}(-1)^{i+j}b_{ij}\det(B_{u,j}^{u,i})\\
&~&-\sum_{1\leqslant i<j\leqslant n}b_{ij}^{2}d_{2}(B_{i,j}^{i,j})\\
&=&(\frac{n^{2}-n}{2})d_{2}(B)+(n-1)\sum_{u=1}^{n}b_{uu}\det(B_{u}^{u})-\sum_{u=1}^{n}b_{uu}(n-1)\det(B_{u}^{u})\\
&~&-\sum_{1\leqslant i<j\leqslant n}b_{ij}^{2}d_{2}(B_{i,j}^{i,j})\\
&=&(\frac{n^{2}-n}{2})d_{2}(B)-\sum_{1\leqslant i<j\leqslant n}b_{ij}^{2}d_{2}(B_{i,j}^{i,j}).
\end{eqnarray*}
\end{proof}

\begin{lemma}\label{lem2.4}
Let $B=(b_{st})_{n\times n}$ and $B_{[ij]}=(b_{st}^{ij})_{n\times n}$ be defined as in Definition \ref{def2.1}, and let $2m$ be the number of non-zero non-diagonal entries of $B$, and let $c$ be the number of 0 entries in the diagonal of $B$. Then
\begin{eqnarray*}
(m-c)d_{2}(B)=\sum_{(i,j)\in I_{1}}d_{2}(B_{[ij]})+\sum_{(i,j)\in I_{2}}b_{ij}^{2}d_{2}(B_{i,j}^{i,j}),
\end{eqnarray*}
where $I_{1}=\{(i,j)|b_{ij}\neq0,1\leqslant i\leqslant j\leqslant n\},~I_{2}=\{(i,j)|b_{ij}\neq0,1\leqslant i< j\leqslant n\}$.
\end{lemma}

\begin{proof}
Let $2k$ be the number of 0-entries not in the diagonal of $B$. Then $2m=n^{2}-n-2k$. By Lemma \ref{lem2.3}, we have
\begin{eqnarray*}
\frac{1}{2}(n^{2}-n)d_{2}(B)
&=&\sum_{1\leqslant i\leqslant j\leqslant n}d_{2}(B_{[ij]})+\sum_{1\leqslant i< j\leqslant n}b_{ij}^{2}d_{2}(B_{i,j}^{i,j})\\
&=&\sum_{\substack{1\leqslant i\leqslant j\leqslant n\\b_{ij}\neq0}}d_{2}(B_{[ij]})+\sum_{\substack{1\leqslant i\leqslant j\leqslant n\\b_{ij}=0}}d_{2}(B_{[ij]})+\sum_{\substack{1\leqslant i< j\leqslant n\\b_{ij}\neq0}}b_{ij}^{2}d_{2}(B_{i,j}^{i,j})\\
&=&\sum_{(i,j)\in I_{1}}d_{2}(B_{[ij]})+\sum_{(i,j)\in I_{2}}b_{ij}^{2}d_{2}(B_{i,j}^{i,j})+(c+k)d_{2}(B).
\end{eqnarray*}
Hence
\begin{eqnarray*}
(m-c)d_{2}(B)=\sum_{(i,j)\in I_{1}}d_{2}(B_{[ij]})+\sum_{(i,j)\in I_{2}}b_{ij}^{2}d_{2}(B_{i,j}^{i,j}).
\end{eqnarray*}

The proof is completed.
\end{proof}

Let $\tau(G;x)$ be defined as in (\ref{equ2}). A routine computation gives rise to the derivative $\tau'(G;x)$ of $\tau(G;x)$:
\begin{eqnarray}\label{equ8}
\tau'(G;x)&=&[\sum_{u=1}^{n}x\det[(xI_{n}-A(G))_{u}^{u}]]'-[\det(xI_{n}-A(G))]'\notag \\
&=&\sum_{u=1}^{n}[\det[(xI_{n}-A(G))_{u}^{u}]+x\sum_{u=1}^{n}\sum_{\substack{i=1\\i\neq u}}^{n}\det[(xI_{n}-A(G))_{i,u}^{i,u}]]\notag \\
&~&-\sum_{u=1}^{n}\det[(xI_{n}-A(G))_{u}^{u}]\notag \\
&=&\sum_{u=1}^{n}\det[(xI_{n}-A(G))_{u}^{u}]+\sum_{u=1}^{n}d_{2}[(xI_{n}-A(G))_{u}^{u}].
\end{eqnarray}

\noindent We obtain a relationship between $\tau(G;x)$ and $\tau'(G;x)$ as follows.
\begin{lemma}\label{lem2.5}
Let $G=(V(G),E(G))$ be an undirected graph without loops or multiple edges, with vertex set $V(G)=\{v_{1},v_{2},\ldots,v_{n}\}$ and edge set $E(G)=\{e_{1},e_{2},\ldots,e_{m}\}$, and let $\tau(G;x)$ be defined as in (\ref{equ2}). Then
\begin{eqnarray*}
(m-n)\tau(G;x)+x\tau'(G;x)=\sum_{v_{s}v_{t}\in E(G)}[\tau(G-v_{s}v_{t};x)+\tau(G-v_{s}-v_{t};x)].
\end{eqnarray*}
\end{lemma}

\begin{proof}
Since $G$ has $m$ edges, $xI_{n}-A(G)$ has $2m$ non-zero non-diagonal entries, and every diagonal entry is not equal to zero. The adjacency matrix of $G-e$ is denoted by $A(G-e)$, and $I_{n}^{(i)}$ is an $n\times n$ identity matrix with the $i$-th diagonal entry being equal to 0. Denote by $x^{*}_{uu}$ the $(u,u)$-entry in matrix $xI_{n}^{(i)}-A(G)$. By Lemma \ref{lem2.4}, we have
\begin{eqnarray}\label{equ9}
m\cdot d_{2}(xI_{n}-A(G))&=&\sum_{i=1}^{n}d_{2}(xI_{n}^{(i)}-A(G))+\sum_{e\in E(G)}d_{2}(xI_{n}-A(G-e))\notag \\
&~&+\sum_{v_{s}v_{t}\in E(G)}(a_{st})^{2}d_{2}[(xI_{n}-A(G))_{s,t}^{s,t}].
\end{eqnarray}
For simplicity of presentation, let $\Delta_{3}=d_{2}(xI_{n}^{(i)}-A(G))$ and $\Delta_{4}=d_{2}(xI_{n}-A(G-e))$ in (\ref{equ9}). By (\ref{equ1}), we have
\begin{eqnarray*}
\Delta_{3}=\sum_{u=1}^{n}x^{*}_{uu}\det[(xI_{n}^{(i)}-A(G))_{u}^{u}]-\det(xI_{n}^{(i)}-A(G)).
\end{eqnarray*}
In $\Delta_{3}$, we note that if $u=i$ then
\begin{eqnarray*}
\sum_{u=1}^{n}x^{*}_{uu}\det[(xI_{n}^{(i)}-A(G))_{u}^{u}]=0,
\end{eqnarray*}
and if $u\neq i$ then
\begin{eqnarray*}
&~&\sum_{u=1}^{n}x^{*}_{uu}\det[(xI_{n}^{(i)}-A(G))_{u}^{u}]\\
&=&\sum_{u=1}^{n}x^{*}_{uu}[\det[(xI_{n}-A(G))_{u}^{u}]-x\det[(xI_{n}-A(G))_{i,u}^{i,u}]].
\end{eqnarray*}
Hence
\begin{eqnarray*}
\Delta_{3}&=&\sum_{\substack{u=1\\u\neq i}}^{n}x^{*}_{uu}\det[(xI_{n}-A(G))_{u}^{u}]-\sum_{\substack{u=1\\u\neq i}}^{n}x^{*}_{uu}x\det[(xI_{n}-A(G))_{i,u}^{i,u}]\\
&~&-\det(xI_{n}-A(G))+x\det[(xI_{n}-A(G))_{i}^{i}]\\
&=&d_{2}(xI_{n}-A(G))-x\det[(xI_{n}-A(G))_{i}^{i}]-x d_{2}[(xI_{n}-A(G))_{i}^{i}],\\
\Delta_{4}&=&\tau(G-e;x).
\end{eqnarray*}
Combining arguments above, we have
\begin{eqnarray}\label{equ10}
&~&\sum_{i=1}^{n}d_{2}(xI_{n}^{(i)}-A(G))\notag \\
&=&\sum_{i=1}^{n}[d_{2}(xI_{n}-A(G))-x\det[(xI_{n}-A(G))_{i}^{i}]-x d_{2}[(xI_{n}-A(G))_{i}^{i}]]\notag \\
&=&n\tau(G;x)-x\sum_{i=1}^{n}\det[(xI_{n}-A(G))_{i}^{i}]-x\sum_{i=1}^{n}d_{2}[(xI_{n}-A(G))_{i}^{i}].
\end{eqnarray}
\begin{eqnarray}\label{equ11}
\sum_{e\in E(G)}d_{2}(xI_{n}-A(G-e))=\sum_{e\in E(G)}\tau(G-e;x).
\end{eqnarray}
Combining (\ref{equ8}), (\ref{equ9}), (\ref{equ10}) and (\ref{equ11}), we have
\begin{eqnarray*}
&~&m\tau(G;x)\\
&=&n\tau(G;x)-x\tau'(G;x)+\sum_{e\in E(G)}\tau(G-e;x)+\sum_{v_{s}v_{t}\in E(G)}a_{st}^{2}d_{2}[(xI_{n}-A(G))_{s,t}^{s,t}].
\end{eqnarray*}
Then
\begin{eqnarray}\label{equ12}
(m-n)\tau(G;x)+x\tau'(G;x)&=&\sum_{e\in E(G)}\tau(G-e;x)+\sum_{v_{s}v_{t}\in E(G)}d_{2}[(xI_{n}-A(G))_{s,t}^{s,t}]\notag\\
&=&\sum_{v_{s}v_{t}\in E(G)}[\tau(G-v_{s}v_{t};x)+\tau(G-v_{s}-v_{t};x)].
\end{eqnarray}

This completes the proof of Lemma \ref{lem2.5}.
\end{proof}

Checking (\ref{equ12}), we find that if $m\neq n$ then the following equation holds.
\begin{eqnarray*}
\tau(G;0)=\frac{1}{m-n}[\sum_{e\in E(G)}\tau(G-e;0)+\sum_{v_{s}v_{t}\in E(G)}\tau(G-v_{s}-v_{t};0)].
\end{eqnarray*}
Hence the differential equation (\ref{equ12}) has an unique solution. Then imply the following result.

\begin{theorem}\label{thm1}
Let $G$ be an undirected graph without loops or multiple
edges, where vertex set $V(G)=\{v_{1},v_{2},\ldots,v_{n}\}$ and edge set $E(G)=\{e_{1},e_{2},\ldots,e_{m}\}$. If $m\neq n$, then $\tau(G;x)$ can be reconstructed from $\{\tau(G-e;x)|e\in E(G)\}\cup\{\tau(G-v_{s}-v_{t};x)|v_{s}v_{t}\in E(G)\}$.
\end{theorem}

\section{A solution of Problem \ref{prob2}}\label{sec3}
In this section, a solution to the Problem \ref{prob2} will be proved. First, we present some notation and lemmas as follows.
\begin{defn}\label{def3.1}
Let $R=(r_{st})_{n\times n}$ be a matrix of order $n$ over the complex field. Define a matrix $R_{ij}=(r_{st}^{ij})_{n\times n}$, where $1\leq i,j\leq n$, and
$$
r_{st}^{ij}=
  \begin{cases}
   r_{st}, & \text{if $(s,t)\neq(i,j)$},\\
   0, & \text{otherwise}.
  \end{cases}
$$
\end{defn}
\begin{example}\label{exam2}
Let $R$ be an $n\times n$ matrix as follows:
$$
R=
\left(
  \begin{array}{cccc}
    r_{11} & r_{12} &  \cdots &  r_{1n} \\
    r_{21} & r_{22} &  \cdots &  r_{2n} \\
    \vdots & \vdots &  \ddots &  \vdots \\
    r_{n1} & r_{n2} &  \cdots &  r_{nn} \\
  \end{array}
\right),
$$
then
$$
R_{12}=
\left(
  \begin{array}{cccc}
   r_{11} & 0 &  \cdots &  r_{1n} \\
   r_{21} & r_{22} &  \cdots &  r_{2n} \\
   \vdots & \vdots &  \ddots &  \vdots \\
   r_{n1} & r_{n2} &  \cdots &  r_{nn} \\
  \end{array}
\right),
R_{11}=
\left(
  \begin{array}{cccc}
    0 & r_{12} &  \cdots &  r_{1n} \\
    r_{21} & r_{22} &  \cdots &  r_{2n} \\
    \vdots & \vdots &  \ddots &  \vdots \\
    r_{n1} & r_{n2} &  \cdots &  r_{nn} \\
  \end{array}
\right).
$$
\end{example}
By Example \ref{exam2}, if $r_{ij}=0$, then $R=R_{ij}$.  Zhang et al. \cite{13} give a result as follows. For any $1\leqslant i, j\leqslant n$,
\begin{eqnarray}\label{equ13}
\det(R_{ij})&=&\det(R)-(-1)^{i+j}r_{ij}\det(R_{j}^{i}).
\end{eqnarray}

\begin{lemma}\label{lem3.2}(\cite{13})
Let $R=(r_{st})_{n\times n}$ and $R_{ij}$ be defined as in Definition \ref{def3.1}. Then the determinant of $R$ satisfies:
\begin{eqnarray*}
(n^{2}-n)\det(R)=\sum_{1\leqslant i,j\leqslant n}\det(R_{ij}).
\end{eqnarray*}
\end{lemma}

There exists a relation between $d_{2}(R)$ and $d_{2}(R_{ij})$ as follows.

\begin{lemma}\label{lem3.3}
Let $R=(r_{st})_{n\times n}$ and $R_{ij}$ be defined as in Definition \ref{def3.1}. Then the second immanant of $R$ and $d_{2}(R_{ij})$ satisfies:
\begin{eqnarray*}
(n^{2}-n)d_{2}(R)=\sum_{1\leqslant i,j\leqslant n}d_{2}(R_{ij}).
\end{eqnarray*}
\end{lemma}

\begin{proof}
Denote by $r_{uu}^{*}$ the $(u,u)$-entry in matrix $R_{ij}$. By (\ref{equ1}) and Lemma \ref{lem3.2}, we have
\begin{eqnarray}\label{equ14}
\sum_{1\leqslant i, j\leqslant n}d_{2}(R_{ij})
&=&\sum_{1\leqslant i, j\leqslant n}\sum_{u=1}^{n}r_{uu}^{*}\det((R_{ij})_{u}^{u})-\sum_{1\leqslant i, j\leqslant n}\det(R_{ij})\notag\\
&=&\sum_{1\leqslant i, j\leqslant n}\sum_{u=1}^{n}r_{uu}^{*}\det((R_{ij})_{u}^{u})-({n^{2}-n})\det (R),
\end{eqnarray}
\begin{eqnarray}\label{equ15}
\sum_{1\leqslant i, j\leqslant n}\sum_{u=1}^{n}r_{uu}^{*}\det((R_{ij})_{u}^{u})=\sum_{i=1}^{n}\sum_{u=1}^{n}r_{uu}^{*}\det((R_{ii})_{u}^{u})
+\sum_{\substack{1\leqslant i, j\leqslant n\\i\neq j}}\sum_{u=1}^{n}r_{uu}^{*}\det((R_{ij})_{u}^{u}).
\end{eqnarray}
For simplicity of presentation, in (\ref{equ15}), define
$$
\Theta_{1}=\sum_{i=1}^{n}\sum_{u=1}^{n}r_{uu}^{*}\det((R_{ii})_{u}^{u})~{\rm and}~
\Theta_{2}=\sum_{\substack{1\leqslant i, j\leqslant n\\i\neq j}}\sum_{u=1}^{n}r_{uu}^{*}\det((R_{ij})_{u}^{u}).~
$$
In $\Theta_{1}$, we note that if $u=i$ then $\Theta_{1}=0$. Assume that $u\neq i$. By (\ref{equ13}), we have
\begin{eqnarray*}
\Theta_{1}&=&\sum_{i=1}^{n}\sum_{u=1}^{n}r_{uu}\det(R_{u}^{u})-\sum_{i=1}^{n}\sum_{u=1}^{n}r_{uu}r_{ii}\det(R_{u,i}^{u,i})\\
&=&(n-1)\sum_{u=1}^{n}r_{uu}\det(R_{u}^{u})-\sum_{i=1}^{n}\sum_{u=1}^{n}r_{uu}r_{ii}\det(R_{u,i}^{u,i}).
\end{eqnarray*}
In $\Theta_{2}$, we note that if $u=i$ then
$$\Theta_{2}=(n-1)\sum_{u=1}^{n}r_{uu}\det(R_{u}^{u}),$$
and if $u=j$ then
$$\Theta_{2}=(n-1)\sum_{u=1}^{n}r_{uu}\det(R_{u}^{u}).$$
Assume that $u\neq i,j$. By (\ref{equ13}), we have
\begin{eqnarray*}
\Theta_{2}&=&\sum_{\substack{1\leqslant i, j\leqslant n\\i\neq j}}\sum_{u=1}^{n}r_{uu}\det(R_{u}^{u}) -\sum_{\substack{1\leqslant i, j\leqslant n\\i\neq j}}\sum_{u=1}^{n}r_{uu}(-1)^{i+j}r_{ij}\det(R_{u,j}^{u,i})\\
&=&(n^{2}-3n+2)\sum_{u=1}^{n}r_{uu}\det(R_{u}^{u}) -\sum_{\substack{1\leqslant i, j\leqslant n\\i\neq j}}\sum_{u=1}^{n}r_{uu}(-1)^{i+j}r_{ij}\det(R_{u,j}^{u,i}).
\end{eqnarray*}
Combining arguments above, we obtain that
\begin{eqnarray}\label{equ16}
&~&\sum_{1\leqslant i, j\leqslant n}\sum_{u=1}^{n}r_{uu}^{*}\det((R_{ij})_{u}^{u})\notag \\ &=&(n^{2}-1)\sum_{u=1}^{n}r_{uu}\det(R_{u}^{u})-\sum_{\substack{1\leqslant i,j\leqslant n\\i,j\neq u}} \sum_{u=1}^{n}r_{uu}(-1)^{i+j}r_{ij}\det(R_{u,j}^{u,i}).
\end{eqnarray}
By substituting (\ref{equ16}) into (\ref{equ14}), we obtain that
\begin{eqnarray*}
&~&\sum_{1\leqslant i, j\leqslant n}d_{2}(R_{ij})\\
&=&(n^{2}-1)\sum_{u=1}^{n}r_{uu}\det(R_{u}^{u})-\sum_{\substack{1\leqslant i,j\leqslant n\\i,j\neq u}} \sum_{u=1}^{n}r_{uu}(-1)^{i+j}r_{ij}\det(R_{u,j}^{u,i})-({n^{2}-n})\det(R)\\
&=&n^{2}d_{2}(R)-\sum_{u=1}^{n}r_{uu}\det(R_{u}^{u})-(n-1)\sum_{u=1}^{n}r_{uu}\det(R_{u}^{u})+n\det(R)\\
&=&n^{2}d_{2}(R)-nd_{2}(R).
\end{eqnarray*}
\end{proof}

By Lemma \ref{lem3.3}, we can obtain a corollary as follows.

\begin{cor}\label{cor3.4}
Let $R=(r_{st})_{n\times n}$ and $R_{ij}$ be defined as in Definition \ref{def3.1}. Suppose that $l$ is the number of non-zero of entries in $R$. Then
\begin{eqnarray*}
(l-n)d_{2}(R)=\sum_{(i,j)\in I}d_{2}(R_{ij}),
\end{eqnarray*}
where $I=\{(i,j)|r_{ij}\neq0,1\leqslant i, j\leqslant n\}$.
\end{cor}

Let $\eta\in\{0,1\}$ and $\alpha\in\{-1,1\}$, and let ~$\overrightarrow{G}=(V(\overrightarrow{G}),~E(\overrightarrow{G}))$ be a digraph without loops or multiple arcs, where vertex set $V(\overrightarrow{G})=\{v_{1},~v_{2},\ldots,v_{n}\}$, and arc set ~$E(\overrightarrow{G})=\{\overrightarrow{e_{1}},~\overrightarrow{e_{2}},\ldots,\overrightarrow{e_{m}}\}$. Define
\begin{eqnarray}\label{equ17}
g(\overrightarrow{G};x)=d_{2} (xI_{n}-\eta D(\overrightarrow{G})-\alpha A(\overrightarrow{G})),
\end{eqnarray}
where $D(\overrightarrow{G})$ and $A(\overrightarrow{G})$ are the vertex in-degree diagonal matrix and the adjacency matrix of $\overrightarrow{G}$, respectively. Suppose that $d^{-}(v_{i})$ is the in-degree of $v_{i}$ of $\overrightarrow{G}$ and $\overrightarrow{e}=(v_{s},v_{t})$. A routine computation gives rise to the derivative $g'(\overrightarrow{G};x)$ of $g(\overrightarrow{G};x)$:
\begin{eqnarray}\label{equ18}
&~&g'(\overrightarrow{G};x)\notag \\
&=&[\sum_{u=1}^{n}(x-\eta d^{-}(v_{u}))\det[(xI_{n}-\eta D(\overrightarrow{G})-\alpha A(\overrightarrow{G}))_{u}^{u}]]'-[\det(xI_{n}-\eta D(\overrightarrow{G})-\alpha A(\overrightarrow{G}))]'\notag \\
&=&\sum_{u=1}^{n}\det[(xI_{n}-\eta D(\overrightarrow{G})-\alpha A(\overrightarrow{G}))_{u}^{u}]+\sum_{u=1}^{n}\sum_{\substack{i=1\\i\neq u}}^{n}(x-\eta d^{-}(v_{i}))\notag\\
&~&\cdot\det[(xI_{n}-\eta D(\overrightarrow{G})-\alpha A(\overrightarrow{G}))_{i,u}^{i,u}]]-\sum_{u=1}^{n}\det[(xI_{n}-\eta D(\overrightarrow{G})-\alpha A(\overrightarrow{G}))_{u}^{u}]\notag \\
&=&\sum_{u=1}^{n}\det[(xI_{n}-\eta D(\overrightarrow{G})-\alpha A(\overrightarrow{G}))_{u}^{u}]+\sum_{u=1}^{n}d_{2}[(xI_{n}-\eta D(\overrightarrow{G})-\alpha A(\overrightarrow{G}))_{u}^{u}].
\end{eqnarray}

We obtain a relationship between $g(\overrightarrow{G};x)$ and $g'(\overrightarrow{G};x)$ as follows.

\begin{lemma}\label{lem3.5}
Let $g(\overrightarrow{G};x)$ be a polynomial defined as in (\ref{equ17}). Then
\begin{eqnarray*}
(m-n)g(\overrightarrow{G};x)+xg'(\overrightarrow{G};x)=\sum_{\overrightarrow{e}\in E(\overrightarrow{G})}g(\overrightarrow{G}-\overrightarrow{e};x).
\end{eqnarray*}
\end{lemma}

\begin{proof}
Denote by $D(\overrightarrow{G}-\overrightarrow{e})$ and $A(\overrightarrow{G}-\overrightarrow{e})$ the vertex in-degree diagonal matrix and the adjacency matrix of $\overrightarrow{G}-\overrightarrow{e}$, respectively. Let $D(\overrightarrow{G})^{(i)}=diag(d^{-}(v_{1}),\ldots,d^{-}(v_{i-1}),0,d^{-}(v_{i+1}),\ldots,d^{-}(v_{n}))$. Denote by $x''_{uu}$ the $(u,u)$-entry in matrix $xI_{n}^{(i)}-\eta D(\overrightarrow{G})^{(i)}-\alpha A(\overrightarrow{G})$. Since $\overrightarrow{G}$ has $m$ arcs, $xI_{n}-\eta D(\overrightarrow{G})-\alpha A(\overrightarrow{G})$ has $m+n$ non-zero entries. By Corollary \ref{cor3.4}, we have
\begin{eqnarray}\label{equ19}
&~&m\cdot d_{2}(xI_{n}-\eta D(\overrightarrow{G})-\alpha A(\overrightarrow{G}))\notag \\
&=&\sum_{i=1}^{n}d_{2}(xI_{n}^{(i)}-\eta D(\overrightarrow{G})^{(i)}-\alpha A(\overrightarrow{G}))+\sum_{\overrightarrow{e}\in E(\overrightarrow{G})}d_{2}(xI_{n}-\eta D(\overrightarrow{G})-\alpha A(\overrightarrow{G}-\overrightarrow{e})).
\end{eqnarray}
For simplicity of presentation, in (\ref{equ19}), define
$$
\Theta_{3}=d_{2}(xI_{n}^{(i)}-\eta D(\overrightarrow{G})^{(i)}-\alpha A(\overrightarrow{G}))~{\rm and}~
\Theta_{4}=d_{2}(xI_{n}-\eta D(\overrightarrow{G})-\alpha A(\overrightarrow{G}-\overrightarrow{e})).
$$
By (\ref{equ1}), we have
\begin{eqnarray*}
\Theta_{3}=\sum_{u=1}^{n}x''_{uu}\det[(xI_{n}^{(i)}-\eta D(\overrightarrow{G})^{(i)}-\alpha A(\overrightarrow{G}))_{u}^{u}]-\det(xI_{n}^{(i)}-\eta D(\overrightarrow{G})^{(i)}-\alpha A(\overrightarrow{G})),
\end{eqnarray*}
\begin{eqnarray*}
\Theta_{4}&=&\sum_{u=1}^{n}(x-\eta d^{-}(v_{u}))\det[(xI_{n}-\eta D(\overrightarrow{G})-\alpha A(\overrightarrow{G}-\overrightarrow{e}))_{u}^{u}]\\
&~&-\det(xI_{n}-\eta D(\overrightarrow{G})-\alpha A(\overrightarrow{G}-\overrightarrow{e})).
\end{eqnarray*}
In $\Theta_{3}$, we note that if $u=i$ then
\begin{eqnarray*}
\sum_{u=1}^{n}x''_{uu}\det[(xI_{n}^{(i)}-\eta D(\overrightarrow{G})^{(i)}-\alpha A(\overrightarrow{G}))_{u}^{u}]=0,
\end{eqnarray*}
and if $u\neq i$ then
\begin{eqnarray*}
&~&\sum_{u=1}^{n}x''_{uu}\det[(xI_{n}^{(i)}-\eta D(\overrightarrow{G})^{(i)}-\alpha A(\overrightarrow{G}))_{u}^{u}]\\
&=&\sum_{u=1}^{n}x''_{uu}[\det[(xI_{n}-\eta D(\overrightarrow{G})-\alpha A(\overrightarrow{G}))_{u}^{u}]-(x-\eta d^{-}(v_{i}))\det[(xI_{n}-\eta D(\overrightarrow{G})-\alpha A(\overrightarrow{G}))_{i,u}^{i,u}]].
\end{eqnarray*}
We obtain that
\begin{eqnarray*}
\Theta_{3}&=&\sum_{\substack{u=1\\u\neq i}}^{n}x''_{uu}\det[(xI_{n}-\eta D(\overrightarrow{G})-\alpha A(\overrightarrow{G}))_{u}^{u}]-\sum_{\substack{u=1\\u\neq i}}^{n}x''_{uu}(x-\eta d^{-}(v_{i}))\\
&~&\cdot\det[(xI_{n}-\eta D(\overrightarrow{G})-\alpha A(\overrightarrow{G}))_{i,u}^{i,u}]-\det(xI_{n}-\eta D(\overrightarrow{G})-\alpha A(\overrightarrow{G}))\\
&~&+(x-\eta d^{-}(v_{i}))\det[(xI_{n}-\eta D(\overrightarrow{G})-\alpha A(\overrightarrow{G}))_{i}^{i}]\\
&=&d_{2}(xI_{n}-\eta D(\overrightarrow{G})-\alpha A(\overrightarrow{G}))-(x-\eta d^{-}(v_{i}))\det[(xI_{n}-\eta D(\overrightarrow{G})-\alpha A(\overrightarrow{G}))_{i}^{i}]\\
&~&-(x-\eta d^{-}(v_{i})) d_{2}[(xI_{n}-\eta D(\overrightarrow{G})-\alpha A(\overrightarrow{G}))_{i}^{i}].
\end{eqnarray*}
In $\Theta_{4}$, we note that if $u=t$ then
\begin{eqnarray*}
&~&\sum_{u=1}^{n}(x-\eta d^{-}(v_{u}))\det[(xI_{n}-\eta D(\overrightarrow{G})-\alpha A(\overrightarrow{G}-\overrightarrow{e}))_{u}^{u}]\\
&=&(x-\eta d^{-}(v_{t}))\det[(xI_{n}-\eta D(\overrightarrow{G})-\alpha A(\overrightarrow{G}))_{t}^{t}],
\end{eqnarray*}
and if $u\neq t$ then
\begin{eqnarray*}
&~&\sum_{u=1}^{n}(x-\eta d^{-}(v_{u}))\det[(xI_{n}-\eta D(\overrightarrow{G})-\alpha A(\overrightarrow{G}-\overrightarrow{e}))_{u}^{u}]\\
&=&\sum_{u=1}^{n}(x-\eta d^{-}(v_{u}))[\det[(xI_{n}-\eta D(\overrightarrow{G}-\overrightarrow{e})-\alpha A(\overrightarrow{G}-\overrightarrow{e}))_{u}^{u}]\\
&~&-\eta\det[(xI_{n}-\eta D(\overrightarrow{G})-\alpha A(\overrightarrow{G}))_{t,u}^{t,u}]].
\end{eqnarray*}
We obtain that
\begin{eqnarray*}
\Theta_{4}
&=&(x-\eta d^{-}(v_{t}))\det[(xI_{n}-\eta D(\overrightarrow{G})-\alpha A(\overrightarrow{G}))_{t}^{t}]\\
&~&+\sum_{\substack{u=1\\u\neq t}}^{n}(x-\eta d^{-}(v_{u}))\det[(xI_{n}-\eta D(\overrightarrow{G}-\overrightarrow{e})-\alpha A(\overrightarrow{G}-\overrightarrow{e}))_{u}^{u}]\\
&~&-\eta\sum_{\substack{u=1\\u\neq t}}^{n}(x-\eta d^{-}(v_{u}))\det[(xI_{n}-\eta D(\overrightarrow{G})-\alpha A(\overrightarrow{G}))_{t,u}^{t,u}]\\
&~&-\det(xI_{n}-\eta D(\overrightarrow{G}-\overrightarrow{e})-\alpha A(\overrightarrow{G}-\overrightarrow{e}))+\eta\det[(xI_{n}-\eta D(\overrightarrow{G})-\alpha A(\overrightarrow{G}))_{t}^{t}]\\
&=&d_{2}(xI_{n}-\eta D(\overrightarrow{G}-\overrightarrow{e})-\alpha A(\overrightarrow{G}-\overrightarrow{e}))-\eta d_{2}[(xI_{n}-\eta D(\overrightarrow{G})-\alpha A(\overrightarrow{G}))_{t}^{t}]\\
&~&-(x-\eta d^{-}(v_{t})+\eta)\det[(xI_{n}-\eta D(\overrightarrow{G}-\overrightarrow{e})-\alpha A(\overrightarrow{G}-\overrightarrow{e}))_{t}^{t}]\\
&~&+(x-\eta d^{-}(v_{t}))\det[(xI_{n}-\eta D(\overrightarrow{G})-\alpha A(\overrightarrow{G}))_{t}^{t}]]\\
&=&d_{2}(xI_{n}-\eta D(\overrightarrow{G}-\overrightarrow{e})-\alpha A(\overrightarrow{G}-\overrightarrow{e}))-\eta d_{2}[(xI_{n}-\eta D(\overrightarrow{G})-\alpha A(\overrightarrow{G}))_{t}^{t}]\\
&~&-\eta\det[(xI_{n}-\eta D(\overrightarrow{G})-\alpha A(\overrightarrow{G}))_{t}^{t}]].
\end{eqnarray*}
Combining arguments above, we have
\begin{eqnarray}\label{equ20}
&~&\sum_{i=1}^{n}d_{2}(xI_{n}^{(i)}-\eta D(\overrightarrow{G})^{(i)}-\alpha A(\overrightarrow{G}))\notag \\
&=&ng(\overrightarrow{G};x)-\sum_{i=1}^{n}(x-\eta d^{-}(v_{i}))\det[(xI_{n}-\eta D(\overrightarrow{G})-\alpha A(\overrightarrow{G}))_{i}^{i}]\notag\\
&~&-\sum_{i=1}^{n}(x-\eta d^{-}(v_{i}))d_{2}\det[(xI_{n}-\eta D(\overrightarrow{G})-\alpha A(\overrightarrow{G}))_{i}^{i}].
\end{eqnarray}
\begin{eqnarray}\label{equ21}
&~&\sum_{\overrightarrow{e}\in E(\overrightarrow{G})}d_{2}(xI_{n}-\eta D(\overrightarrow{G})-\alpha A(\overrightarrow{G}-\overrightarrow{e}))\notag\\
&=&\sum_{\overrightarrow{e}\in E(\overrightarrow{G})}g(\overrightarrow{G}-\overrightarrow{e};x)-\eta\sum_{t=1}^{n}d^{-}(v_{t}))d_{2}[(xI_{n}-\eta D(\overrightarrow{G})-\alpha A(\overrightarrow{G}))_{t}^{t}]\notag\\
&~&-\eta\sum_{t=1}^{n}d^{-}(v_{t}))\det[(xI_{n}-\eta D(\overrightarrow{G})-\alpha A(\overrightarrow{G}))_{t}^{t}].
\end{eqnarray}
Combining (\ref{equ18}), (\ref{equ19}), (\ref{equ20}) and (\ref{equ21}), we have
\begin{eqnarray}\label{equ22}
mg(\overrightarrow{G};x)=ng(\overrightarrow{G};x)-xg'(\overrightarrow{G};x)+\sum_{\overrightarrow{e}\in E(\overrightarrow{G})}g(\overrightarrow{G}-\overrightarrow{e};x).
\end{eqnarray}
The proof is completed.
\end{proof}

By (\ref{equ22}), we note that if $m\neq n$ then the following equation holds.
\begin{eqnarray*}
g_{i}(\overrightarrow{G};0)=\frac{1}{m-n}\sum_{\overrightarrow{e}\in E(\overrightarrow{G})}g_{i}(\overrightarrow{G}-\overrightarrow{e};0).
\end{eqnarray*}
Thus the differential equation (\ref{equ22}) has an unique solution. This implies the following result.

\begin{theorem}\label{thm2}
Let $\overrightarrow{G}=(V(\overrightarrow{G}),~E(\overrightarrow{G}))$ be a digraph having no loops and no multiple arcs, with vertex set $V(\overrightarrow{G})=\{v_{1},~v_{2},\ldots,v_{n}\}$ and arc set $E(\overrightarrow{G})=\{\overrightarrow{e}_{1},~\overrightarrow{e}_{2},\ldots,\overrightarrow{e}_{m}\}$. If $m\neq n$, then $g_{i}(\overrightarrow{G};x)$ can be reconstructed from $\{g_{i}(\overrightarrow{G}-\overrightarrow{e};x)|e\in E(\overrightarrow{G})\}$ for $i=1,2,3$.
\end{theorem}

\section{Summary}\label{sec4}

 In this paper, we proved that if $|V(G)|\neq |E(G)|$, then the second immanantal polynomial $\tau(G;x)$ of the adjacency matrix of undirected graph $G$ can be reconstructed from the second immanantal polynomials of the adjacency matrix of all subgraphs in $\{G-v_{s}v_{t},G-v_{s}-v_{t}|v_{s}v_{t}\in E(G)\}$. And if $|V(\overrightarrow{G})|\neq |E(\overrightarrow{G})|$, then the second immanantal polynomial $g_{i}(\overrightarrow{G};x)(i=1,2,3)$ of the adjacency matrix and (signless) Laplacian  matrix of digraph $\overrightarrow{G}$ can be reconstructed from the second immanantal polynomials of the adjacency matrix and (signless) Laplacian  matrix of all subgraphs in $\{\overrightarrow{G}-\overrightarrow{e}|\overrightarrow{e}\in E(\overrightarrow{G})\}$, respectively.

To conclude, we have the following problems, which are to be investigated in the future.

1. Can the second immanantal polynomials of the (signless) Laplacian  matrix of undirected graph satisfy the edge reconstruction conjecture?

2. For each hook partition $\lambda$ of $n$, can the immanant function $d_{\lambda}$ satisfy the edge reconstruction conjecture?

\noindent{\bf Data Availability}\\
Data from previous studies were used to support this study.
They are cited at relevant places within the text as references.

\noindent{\bf Conflicts of Interest}\\
The authors declare that they have no conflicts of interest.

\nocite{*}

%
%
%

\end{document}